\documentclass{article}
\usepackage[cp1251]{inputenc}   

\usepackage{hyperref}
\usepackage{mathtools}
\usepackage{amssymb}
\usepackage{amsmath}
\usepackage{amsthm}
\usepackage{graphicx}

\newtheorem{lemma}{Lemma}

\newtheorem{theorem}{Theorem}
\newtheorem{conclusion}{Corollary}

\usepackage{tikz}
\usetikzlibrary{positioning}
\tikzset{main node/.style={circle,fill=blue!20,draw,minimum size=0.5cm,inner sep=0pt}}
\date{}
\author{E.Yu.~Lerner\footnote{This work was supported by the Russian Science Foundation (project N. 24-21-00158).}}
\title{The Heawood approach to Tait colorings and defining vertex sets}
\begin{document}
\maketitle
\begin{abstract}
Given a simple biconnected planar cubic graph, we associate each its vertex among $2n$ ones with the so-called spin, i.e., a variable which takes on values $\pm 1$. P.~J.~Heawood has proved that a Tait coloring, accurate to the choice of a color for one edge, is equivalent to the choice of spin values so as to make the sum of these value at vertices of any face be a multiple of~3. We treat faces, which satisfy this condition, as {\it proper}. 
The condition that guarantee the propriety of faces define a system of linear equations (SLE) with respect to variables, which take on nonzero values in the field~${\mathbb F}_3$. We say that a set of vertices is {\it defining} if values of spins of these vertices uniquely define values of the rest spins. In particular, so is the set of vertices which correspond to all free variables of the SLE. We actualize the approach proposed by P.~J.~Heawood by proposing a geometric proof of the fact that for a non-bipartite graph the rank of the SLE equals $n+1$. Moreover, we also geometrically describe the necessary condition for the minimality of the defining set. This implies that in the case of a non-bipartite graph there exist defining subsets consisting of $n-1$ vertices. As a simple corollary, we conclude that the number of Tait colorings in this case does not exceed $3\cdot 2^{n-1}$. Though this estimate is not exact, it is by half better than the known one. We also prove that the number of Tait colorings for a graph $CL_n$, which is bipartite for even $n$ and non-bipartite for an odd one, equals $2^n+8$ and $2^n-2$, correspondingly\footnote{See the Remark at the end part of the paper.}.
\end{abstract}



\section{Introduction, main results}

Given a {\it simple biconnected planar cubic} graph $G=(V,E)$, we assume that the number of edges in it equals $3n$ and, correspondingly, the number of vertices equals $2n$, $n=2,3,\ldots$. The {\it Tait coloring} is a coloring of edges that enter in the set~$E$ in 3 colors in such a way that all edges with a common vertex are colored differently. Establishing the existence of such a coloring for any graph~$G$ is equivalent to proving the Four Color Theorem. Let us denote the number of various Tait colorings for the graph $G$ by the symbol $\chi’_3(G)$.

There are many papers devoted to Tait colorings. Let us briefly mention only works which are related to this paper and to the algebraic approach to constructing Tait colorings. In 1898, P.~J.~Heawood proved that constructing a Tait coloring is equivalent to finding a nonzero solution to a certain system of equations in the field~${\mathbb F}_3$. Moreover, he has stated inequalities for the rank of this system~(\cite{heawood}) and studied (in 1940s) various particular cases of this system. In 1967, in the book~(\cite{Ore}) devoted to the Four Color Problem, O.~Ore described in detail all known results, in particular, those obtained by P.~J.~Heawood. E.~G.~Belaga has generalized the approach proposed by P.~J.~Heawood for arbitrary Riemann surfaces, proposed its geometric interpretation. 
Following P.~J.~Heawood, E.~G.~Belaga has established that the fact of whether the graph is bipartite or not depends 
on the value (among two possible ones) of the rank of the Heawood system.

In pioneer papers~(\cite{DiscrAn,CombAs,POMI}), Yu.~V.~Matiyasevich introduces the notion of the graph discriminant and establishes the connection of its coefficients with the value $\chi’_3(G)$. Basing on this connection, he, in particular, expresses the value $\chi’_3(G)$ in terms of the conditional probability of the so-called parity of directions of edges of the graph $G$ with respect to their equivalence modulo~3. In the rather recent paper~(\cite{karpov}), D.~V.~Karpov introduces the notion of the defining set of edges of the graph~$G$ (the set, whose coloring uniquely defines the Tait coloring of the rest edges of the graph~$G$). He proves that the graph~$G$ has a defining set consisting of $n$ edges. This property has allowed Karpov to prove the estimate $\chi’_3(G)\leq  9\cdot 2^{n-2}$.

This paper is to some extent inspired by the latter work, but here we use a somewhat modernized variant of the classical Heawood approach. It seems that our approach is rather simple and clear.

Let us exactly describe results of this paper. Numerate colors used in the Tait coloring as $0,1,2$. We assume that they represent elements of the field ${\mathbb F}_3$. Denote its {\it nonzero elements} as ${\mathbb F}^*_3$:  ${\mathbb F}^*_3=\{1,2\}\equiv \{1,-1\}$. Evidently, by changing colors of all edges in a Tait coloring by a cyclic shift, namely, $x\to x+a$, $a\in{\mathbb F}_3$, we again get a Tait coloring. In other words, the set of all Tait colorings falls onto the {\it set of three-element classes}~$K(G)$ of colorings, which are obtainable from one another by a cyclic shift of colors.

Let us associate each vertex $v$, $v\in V$, of the graph under consideration with a spin (a variable) $\sigma(v)$ which takes on values in ${\mathbb F}^*_3$. Fix a certain set of all spin values. 
Let us call this set {\it proper for a face~$f$}, if
\begin{equation}
\label{eq:main_eq}
\sum_{v\in f} \sigma(v)=0.
\end{equation}
In such a case, for convenience, we will often use the term {\it a proper face}. 

\begin{theorem}[Heawood, 1898]
\label{th:heawood}
There exists a biunique correspondence between $K(G)$ and all possible sets of values of spins which are proper for all faces.
\end{theorem}

We understand the {\it Heawood vector} as a set of nonzero spin values, which is proper for all faces (see \cite{belaga1976,belaga1999}).
We say that a set of vertices, whose spin values uniquely define values of the rest spins in the Heawood vector, is {\it defining} (analogously to the definition of the defining set of edges proposed in~\cite{karpov}). In addition, we say that a defining set is {\it minimal}, if by deleting any vertex in it we turn it into a non-defining set.

If variables $\sigma(v)$ can take on any values in ${\mathbb F}_3$, then conditions~\eqref{eq:main_eq} define a system of $n+2$ linear homogeneous equations (SLE) with $2n$ variables. Indices of free variables in the SLE represent a defining set, but not necessarily the minimal one. In what follows, we study the structure of the set of indices of free variables.

By summing up all equations we get the equality $0=0$. Therefore, we can exclude at least one equation from the SLE. Let the symbol $F_{in}$ stand for the set of all internal faces of the graph~$G$. We treat the system of equations~\eqref{eq:main_eq} as the {\it main} SLE, only if $f\in F_{in}$.

Consider an arbitrary linear combination of equations in the main SLE. We treat the set of vertices that define indices of variables that enter in this combination with nonzero coefficients as the {\it support} of the linear combination. Let us describe all possible supports of the main system.

Fix an arbitrary cycle $C_0$ in the graph~$G$. It divides the plane onto two domains, namely, the external and internal ones. The latter domain contains only the subset of faces that enter in $F_{in}$. Denote the set of vertices of these faces without $V(C_0)$ by $V_{in}(C_0)$. Denote by $C_1, \ldots, C_k$ a set of disjoint cycles, $V(C_i)\subseteq V_{in}(C_0)$, $i=1,\ldots,k$; here the number $k$ of such cycles can take on any value, including zero.

In each cycle $C_i$, $i=0,\ldots,k$, let us choose an even number of vertices. Denote the set of all chosen vertices by the symbol $W$. Let $|W|=2m$. Consider the set of disjoint paths $P_j$, $j=1,\ldots,m$, all whose endpoints belong to $W$, while $V(P_j)\subseteq (V_{in}(C_0)\setminus \cup_{i=1}^kV(C_i))\cup W$ (i.e., all internal vertices of these paths belong to the domain located ``between'' the external cycle $C_0$ and internal cycles $C_i$, $i=1,\ldots,k$).

We understand a {\it zebra} as the set of vertices $\left(\cup_{i=0}^{k} V(C_i) \cup_{j=1}^{m} V(P_j)\right)\setminus W$, and we do the set $\left(V(C_0)\cup V_{in}(C_0)\right) \setminus \cup_{i=1}^k V_{in}(C_i)$ as the {\it zebra body}.

\begin{lemma}
\label{lem:zebra_why}
Any support of the main SLE represents the union of zebras with disjoint bodies.
\end{lemma}

\begin{lemma}
\label{lem:zebra1}
A zebra can coincide with the empty set only if the graph~$G$ is a bipartite one.
\end{lemma}

Lemmas~\ref{lem:zebra_why} and~\ref{lem:zebra1} imply the following theorem.

\begin{theorem}
\label{th:system}
Let a graph~$G$ be not bipartite. In this case, all equations in the main SLE are linearly independent (the rank of the main SLE equals $n+1$). Variables in a certain set are linearly dependent if and only if their indices are defined by a set of vertices containing a zebra. In particular, any set of $n$ variables possesses this property.
\end{theorem}

Actually, Heawood was aware of the first proposition of this theorem (see also~\cite{belaga1976,belaga1999}). It was proved in a less constructive but rather simple way by making use of properties of the rank of the system of equations, whose matrix is transposed with respect to the matrix of the system under consideration, and those of the classical triangulation graph dual to the considered graph~$G$. It was also noticed that in a bipartite case the rank of the main SLE equals~$n$. Note that the specificity of a bipartite case (see, for example, \cite[Proposition~6.4.2]{diestel}, this proposition 
was also stated by to P.~J.~Heawood) consists in the fact that in this and only this case, a cubic biconnected graph has an everywhere nonzero 3-flow. This means (see~\cite[Theorem 6.5.3]{diestel}) that in this and only this case, the dual triangulated graph has a proper vertex 3-coloring.

Let us return to studying sets of vertices. We can restate Theorem~\ref{th:system} as follows.

\begin{theorem}
\label{th:defset}
Indices of free variables in the main SLE for a non-bipartite graph form defining sets consisting of $n-1$ vertices. One can calculate values of the rest variables from the SLE; if they appear to be nonzero, then they define a Heawood vector. If a certain set of vertices of a non-bipartite graph contains a zebra, then it is not the minimal defining set. In particular, any set consisting of $n$ vertices of a non-bipartite graph is not the minimal defining set.
\end{theorem}

As a corollary of this assertion, we obtain the following inequality with respect to the number of Tait colorings.

\begin{conclusion}
\label{concl}
Any simple biconnected non-bipartite planar cubic graph $G$ with $2n$ vertices satisfies the estimate
\begin{equation}
\label{eq:chiGneq}
\chi’_3(G)\leq  3\cdot 2^{n-1}.
\end{equation}
\end{conclusion}

According to results of computer-based calculations for small dimensions, the greatest value of $\chi’_3(G)$ with a fixed number of vertices $2n$ is attained for the {\it circular ladder graph} $CL_n$ that represents the direct product of a cycle $C_n$ consisting of $n$ vertices, $n>2$, and a single-edge path $P_2$. Evidently, this graph is bipartite for an even number~$n$ and non-bipartite for an odd one. Let us prove the following result with the help of Theorem~\ref{th:heawood}.

\begin{theorem}
\label{th:Cln}
The following formula is valid:
$$
\chi’_3(CL_n)=
\begin{cases}
2^n+8, \mbox{if $n$ is even;}\\
2^n-2, \mbox{if $n$ is odd.}
\end{cases}
$$
\end{theorem}

Therefore, the limit value of the obtained estimate~\eqref{eq:chiGneq} cannot be improved more than by half (in a general case, one can try to improve it up to the bound $2^n+\text{const}$\footnote{This result was obtained by M.~P.~Ivanov, see Remark at the end part of the paper.}).

Let us illustrate notions used in this paper with the help of one simple biconnected planar cubic graph with 6 vertices, namely, i.e., the graph $CL_3$ (see Fig.~\ref{pic:1}).
Since it is three-connected, the set of its faces is independent of the graph embedding.

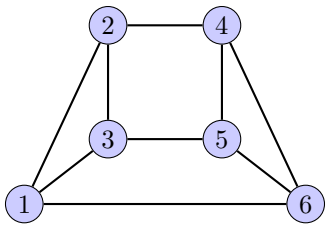
\begin{figure}[h]
\begin{center}
\begin{tikzpicture}
    \node[main node] (2) {$2$};
    \node[main node] (4) [right = 1cm  of 2]  {$4$};
    \node[main node] (3) [below = 1cm  of 2] {$3$};
    \node[main node] (5) [right = 1cm  of 3] {$5$};
    \node[main node] (1) [below left = 0.5cm and 0.75cm of 3]  {$1$};
    \node[main node] (6) [below right = 0.5cm and 0.75cm of 5]  {$6$};

    \path[draw,thick]
    (1) edge node {} (2)
    (1) edge node {} (3)
    (1) edge node {} (6)
    (2) edge node {} (3)
    (2) edge node {} (4)
    (3) edge node {} (5)
    (4) edge node {} (5)
    (4) edge node {} (6)
    (5) edge node {} (6)
        ;
\end{tikzpicture}
\caption{\label{pic:1} The graph $CL_3$ and its faces.}
\end{center}
\end{figure}

One can easily prove that for this graph there exist two Heawood vectors
$$
(\sigma(1),\ldots, \sigma(6))=(1,1,1,-1,-1,-1)\quad \mbox{and }
(-1,-1,-1,1,1,1).$$ 
These two vectors correspond to two classes of Tait colorings. In the first case, numbers of colors ${0,1,2}\in{\mathbb F}_3$ of edges, which are incident to vertices $1,2,3$, increase by $1$ when bypassing counterclockwise. At the same time, colors of edges, which are incident to vertices $4,5,6$, decrease by $1$ when bypassing counterclockwise. In the second case, the opposite is true. One can easily define Tait colorings from these conditions, no other coloring exists.

Let us make sure that the set of vertices $\{1, 6\}$ forms a zebra. Consider the cycle $C_0=(124653)$; it confines a domain composed of faces $(123)$, $(2453)$, and $(465)$, $V_{in}(C_0)=\emptyset$; let us construct paths $(2,3)$ and $(4,5)$. 
Let us make sure that the set of vertices $\{2,4\}$ also forms a zebra. To this end, choose the cycle $(123546)$ and construct paths $(1,3)$ and $(5,6)$.
The set of vertices $\{3,5\}$ also forms a zebra. In this case, we choose the cycle $C_0=(1246)$, $V_{in}(C_0)=\{3,5\}$, and paths $(1,3,2)$ and $(4,5,6)$. There exists no other two-element zebra.

Any subset consisting of three vertices of this graph contains a zebra. In particular, each vertex set $\{1,2,5\}$ and $\{3,4,6\}$ is a zebra; we can get them by using the same cycle $C_0=(1246)$ and considering only one path among two ones, namely, $(4,5,6)$ and $(1,3,2)$. Each vertex set $\{1,3,4\}$ and $\{2,5,6\}$ is a zebra.
For example, $\{1,3,4\}$ is a zebra, because it represents the difference of vertices that enter in the cycle $(13546)$ and the path $(5,6)$. Finally, each of sets $\{1,4,5\}$ and $\{2,3,6\}$ is a zebra; we can obtain them from cycles $C_0=(12453)$ and $C_0=(24653)$, correspondingly. All the rest three-element subsets of vertices either contain two-element zebras $\{1,6\}$, $\{2,4\}$, $\{3,5\}$, or coincide with one of cycles $(123)$ or $(465)$.

Therefore, in this case, the set of free variables for the main SLE consists of all two-element sets indexed with graph vertices different from $\{1,6\}$, $\{2,4\}$, $\{3,5\}$. Correspondingly, the defining set considered in Theorem~\ref{th:system} is any of sets consisting of two vertices which differs from $\{1,6\}$, $\{2,4\}$, $\{3,5\}$. In addition, no three-element set of vertices is minimal defining.

Note that in Theorem~\ref{th:system} one does not consider the question of the minimality of defining sets consisting of $n-1$ elements, which correspond to indices of free variables in the main SLE. 
According to the adduced example, they are not necessarily minimal, namely, in the given example, the Heawood vector is uniquely defined from the value of any spin. In Conclusion, we discuss the exactness of estimate~\eqref{eq:chiGneq} in connection with minimal defining sets.

The next part of the paper has the following structure.
In Section~2.1, we give a rather natural, in our opinion, short proof of the Heawood theorem (cf. with the proof of this result proposed in~\cite{heawood,belaga1999}). The idea of our proof is close to that proposed in book~\cite{Ore}. In Section~2.2, we prove Lemma~\ref{lem:zebra_why}.
Section~2.3 is devoted to the proof of Lemma~\ref{lem:zebra1}, theorems~\ref{th:system} and~\ref{th:defset}, and Corollary~\ref{concl}.
In Section~2.4, we prove Theorem~\ref{th:Cln}.
In Conclusion, we summarize the obtained results and discuss unsolved questions.

\section{Proofs}
\subsection{The correspondence between Tait colorings and Hea\-wood vectors.}
The proof of Theorem~\ref{th:heawood} proposed by us is based on two very well-known facts. The first one is related to the correspondence between proper colorings and everywhere nonzero flows (see, for example, \cite[chapter 6]{diestel}).

Consider a standard proper coloring of vertices of a simple graph~$H=(V,E)$ in colors that belong to the set ${\mathbb Z}_k$, i.e., the {\it additive group of integer modulo~$k$}. Let the symbol $c(v)$ denote the color of the vertex~$v$, $v\in V$. Note that the set of proper colorings falls onto $k$-element classes of colorings, which can be obtained one from another by a cyclic shift of colors.

Let $G’=(V',E')$ be a pseudograph. Let us arbitrarily direct its edges $E'$. Recall that an {\it everywhere nonzero $k$-flow} in a graph~$G’$ is a correspondence between each directed edge $e \in E'$ with a nonzero element of the considered group $g(e)$ such that for any vertex $v\in V'$ the sum $g(e)$ calculated over all edges $e$ originating at $v$ equals the sum $g(e)$ calculated over all edges $e$, ending at~$v$. We call this requirement the {\it balance condition} at the vertex~$v$.

There exists a biunique correspondence between classes of proper colorings of a planar graph~$H$ and everywhere nonzero flows in the dual graph $H^*$~\cite{tutte1953}. We can describe the idea of this correspondence more simply in the following way. Assume that an edge $e^*$ of the graph $H^*$ is dual to the edge~$e=xy$ of the graph~$H$, while the vertex~$x$ is located to the left of the edge~$e^*$, and the vertex $y$ is located to the right of the edge~$e^*$, taking into account its directionality. Put $g(e^*)=c(x)-c(y)$ (evidently, this difference is independent of the choice of a representative of the class of equivalent colorings).

The second fact is related to the structure of the graph that is dual to the medial graph of a planar graph~$G=(V,E)$. Recall that the set of vertices of the {\it medial graph} $H$ is formed by edges $E$ of the initial graph, which are incident if and only if they are neighboring edges in one and the same face.
The graph $H^*$ dual to it represents a bipartite graph, which is called a {\it double graph} (see~\cite[Theorem 3.5.4.]{Ore}, in the book by Ore it is called a {\it radial} graph). Its vertices $V^*\cup F^*$ correspond to vertices in the set $V$ of the initial graph~$G$ and its faces~$F$. These two subsets of vertices form two parts of the graph~$H^*$. There exists an edge between $v^*$, $v^*\in V^*$, and $f^*$, $f^*\in F^*$ if and only if the vertex~$v$ belongs to the face~$f$.

\begin{proof}[Proof of Theorem~\ref{th:heawood}] By definition, a Tait coloring of a planar cubic graph is a proper coloring of its medial graph (usually, in such cases, one means the line graph, but for a cubic graph these two notions coincide). Consequently, there exists a biunique correspondence between Tait colorings and everywhere nonzero flows on a double graph.

Let us direct all edges of this bipartite graph from vertices in $V^*$ to those in $F^*$. Then each its vertex $v^*\in V^*$ is the origin of 3 edges $e_1$, $e_2$, $e_3$, while $g(e_i)\in {\mathbb F}^*_3$ and $\sum_{i=1}^3 g(e_i)=0$. Therefore, $ g(e_1)=g(e_2)=g(e_3)\in\{1,-1\}$.

Let us establish a biunique correspondence between everywhere nonzero flows in a  double graph and sets of spins $\sigma(v)$, $\sigma(v)\in\{1,-1\}$. Put $\sigma(v)=g(e_i)$, where $e_i$ is an (arbitrary) edge originating at the corresponding vertex $v^*$ of the double graph. For establishing the biunique correspondence, it suffices to additionally impose the balance condition for the nonzero flow in the double graph for vertices $f^*\in F^*$. In terms of~$\sigma(v)$ this condition takes form~\eqref{eq:main_eq}, $f\in F$.
\end{proof}

Note that if the planar cubic graph $G$ itself is bipartite, then one can easily find the Heawood vector. To this end, it suffices to put $\sigma(v)=1$, if $v$ belongs to the first part of $V(G)$, and do $\sigma(v)=-1$ otherwise. However, the bipartite case of a cubic graph~$G$ is not interesting relative to the problem of Tait colorings. As we mention in Introduction, in the bipartite case, the graph $G$ has an everywhere nonzero 3-flow and, correspondingly, a proper coloring of vertices of the graph~$G^*$, which is dual to $G$, in three colors. We can also easily construct the Tait coloring without the use of the Heawood theorem. A simple way to do this consists in coloring edges in the perfect match of the graph~$G$ in color~0, and doing the rest edges, which represent the set of cycles of an even length, alternately, in colors~1 and~2.

\subsection{Supports of all possible linear combinations of equ\-ations in the main SLE.}
\begin{proof}[Proof of Lemma~\ref{lem:zebra_why}]
Each equation in the SLE corresponds to its face, whose interiority in the planar embedding of the graph consists of points in ${\mathbb R}^2$ confined by a closed Jordan curve. For convenience, we visualize the linear combination by ``coloring’’ these points in two distinct colors in dependence of whether the corresponding equation enters in the considered linear combination with the coefficient $+1$ or $-1$. If the equation does not enter in the linear combination, then the interiority of the face remains non-colored. As a result, we get a colored 
bounded part of the plane, which falls onto several connected colored domains. Let us consider one of them.

By definition, the external boundary of a domain represents some cycle $C_0$. At the same time, the internal boundary of the domain (if exists) represents the set of disjoint cycles $C_1, \ldots, C_k$, which belong to $V_{in}(C_0)$. Let us ``merge’’ the neighboring faces of the same color and consider only the boundary between differently colored domains. In what follows, for convenience, we treat the whole union of a connected colored set of domains (the subset of points in ${\mathbb R}^2$ under consideration) as the {\it zebra body}.

Boundaries between differently colored domains in the zebra body represent disjoint paths of the graph which begin and end at vertices $C_i$, $i=0,\ldots,k$. Moreover, these paths have no other common vertices with these cycles except for their endpoints.

However, disjoint paths are not necessarily domain boundaries. To this end, it is necessary that one should be able to color the zebra body in two colors. In order to establish the necessary and sufficient conditions for this property, we connect one more planar graph with the zebra body divided by mentioned paths onto domains. Vertices of this multigraph $G’$ are domains obtained after constructing paths, while the edge in the graph~$G’$ corresponds to the common path between two domains. As is known, the graph $G’$ is bipartite, if and only if the dual to it pseudograph~$(G’)^*$ is the Euler one. Its vertices are cycles $C_i$, and edges are paths which connect these cycles. Therefore, if the symbol $W$ stands for the set of endpoints of our paths, then the zebra body is two-colorable if and only if the number of vertices in sets $W\cap V(C_i)$, $i=0,\ldots,k$m is even.

Thus, we have established a biunique correspondence between all possible linear combinations of equations in the main SLE and colored in two colors disjoint bodies of zebras mentioned in the Lemma assumption. Let us now study properties of the support of the corresponding linear combination. As above, we restrict ourselves to considering vertices which belong to one zebra body. Recall that the symbol $P_j$, $j=1,\ldots,m$, stands for the set of mentioned paths. Put 
$Z_1=\cup_{j=0}^k V(C_i)\setminus W$ and
$Z_2=\cup_{j=1}^m V(P_j)\setminus W$.

Each vertex $v$ under consideration belongs to one of the following four categories:\\
1) $v\in W$;\\
2) $v\in Z_1$;\\
3) $v\in Z_2$;\\
4) the rest internal vertices of the zebra body.

Evidently, in a cubic planar graph each vertex belongs to 3 faces. Therefore, in the latter case under consideration, the vertex corresponds to the index of the variable that enters in 3 summed equations in the main SLE with one and the same coefficient. This variable does not enter in the sum in the field ${\mathbb F}_3$. In the first case, the corresponding variable enters in 2 equations in the SLE and has opposite signs in the sum, therefore it neither enters in the linear combination. In case~2), the vertex $v$ corresponds to the index of the variable that enters in 1 or 2 equations in the SLE with one and the same coefficient. Finally, in case~3), the corresponding variable enters in 3 equations with nonzero coefficients $\pm 1$, while coefficients do not necessarily have one and the same sign. Therefore, the support of the linear combination in the body of the zebra is the set $Z_1\cup Z_2$, which was to be proved.
\end{proof}

\subsection{The rank of the matrix of the main SLE and defining sets of vertices.}
\begin{proof}[Proof of Lemma~\ref{lem:zebra1}]
Assume that a zebra is the empty set of vertices. In this case, the set $W$ of endpoints of paths coincides with $\cup_{i=0}^k V(C_i)$, while paths themselves have no internal vertices, i.e., they consist of one edge. Note that the set of edges of paths and cycles contains all 3 edges of any vertex $w\in W$. This means that the graph induced by the vertex set $W$ (i.e., the graph with the vertex set $W$ and all edges of the graph~$G$, whose endpoints belong to $W$) coincides with the graph $G$.

We denote the set of edges $\cup_{i=0}^k E(C_i)$ by the symbol $E_C$ and we do the set of edges $E(G)\setminus E_C$ by $E_P$. 
In other words, $E_P=\cup_{j=1}^m E(P_j)$.
Let us consider the counterclockwise bypass of edges of an arbitrary face of the graph $G$.

In this bypass, any edge $e_C$ in $E_C$ should be followed by an edge $e_P$ in $E_P$, i.e., we cannot go further in this cycle, because ``we have to turn to the left’’. Visa versa, an edge $e_P$ in $E_P$ should be followed by an edge in $E_C$, because edges in $E_P$ are not adjacent. Consequently, any face of the graph~$G$ is formed by an even number of edges, therefore, the graph~$G$ is bipartite.
\end{proof}

Note that results obtained in papers~\cite{heawood,belaga1976} in view of Theorem~\ref{th:system} imply that the bipartite character of a graph is not only the necessary condition for the existence of the empty zebra, but also a sufficient one. 

\begin{proof}[Proof of Theorem~\ref{th:system}] If a system consisting of $(n+1)$ equations is degenerate, then there exists a linear combination with nonzero coefficients of equations of the SLE, which leads to the tautology $0 = 0$. In this case, the zebra coincides with the empty set, which contradicts Lemma~\ref{lem:zebra1}. The rest assertions of the theorem are also evident, because due to properties of the support of the main SLE one of variables, whose index belongs to the support, can be expressed in terms of the rest variables with the corresponding indices.
\end{proof}

\begin{proof}[Proof of Theorem~\ref{th:defset}]
If in some defining set of vertices a non-bipartite graph contains a zebra, then, as was mentioned above, we can express one of variables, whose index belongs to the zebra, in terms of the rest ones. Consequently, by deleting this variable we also get a defining set of vertices of the initial graph. This contradicts the minimality of the initial defining set. The rest assertions also follow from definitions.
\end{proof}

\begin{proof}[Proof of Corollary~\ref{concl}]
The above theorem implies that there exists a defining set consisting of $n-1$ vertices. We can choose values of its spins in $2^{n-1}$ ways. Values of all the rest spins in the Heawood vector (if they exist) are defined uniquely. Therefore, the number of Heawood vectors does not exceed $2^{n-1}$. Consequently, by Theorem~\ref{th:heawood} the number of all possible Tait colorings does not exceed $3\times 2^{n-1}$.
\end{proof}

\subsection{The number of Tait colorings for a circular ladder graph.}
\begin{proof}[Proof of Theorem~\ref{th:Cln}]
A circular ladder graph has $n$ faces formed by 4 vertices and 2 faces formed by $n$ vertices. Denote vertices of the latter faces by symbols $v_1,\ldots,v_n$ and $w_1,\ldots,w_n$, synchronizing the numeration of vertices of these faces so as to make faces formed by 4 vertices take the form $(v_i,w_i,w_{i+1},v_{i+1})$, $i=0,\ldots,n-1$, where $v_0\equiv v_n$ and $w_0\equiv w_n$. For example, in the case of the graph $CL_3$ shown in Fig.~\ref{pic:1}, one should replace vertex numbers 1,2,3,4,5,6 with $v_1, v_2, v_3, w_2, w_3, w_1$, correspondingly.

Let us calculate the number of Heawood vectors. Let us first consider the case, when there exists a pair of spin values such that $\sigma(v_i)=\sigma(w_i)$ for some $i$. 
Then the propriety conditions for faces formed by 4 vertices imply the equality $\sigma(v_{i+1})=\sigma(w_{i+1})=-\sigma(v_i)$ and so on. We conclude that with even~$n$ there exist exactly 2 such Heawood vectors with opposite signs. With odd $n$, no such Heawood vector exists.

Let us now consider the case, when $\sigma(v_i)=-\sigma(w_i)$ for each $i=1,\ldots,n$. 
Then the propriety conditions for faces formed by 4 vertices is fulfilled automatically. The problem is reduced to finding the number of sets of spins which are proper for the face $(v_1,\ldots,v_n)$. In other words, we need to calculate the number of sequences $(x_1,\ldots,x_n)$, $x_i\in {\mathbb F}^*_3$, such that $\sum_{i=1}^n x_i=0$. By induction with respect to $n$ we can easily prove that the number of such sequences for even~$n$ equals $(2^n+2)/3$ (and, concurrently, the number of sequences such that $\sum_{i=1}^n x_i=1,2$, equals $(2^n-1)/3$); at the same time, for odd~$n$ the number of such sequences equals $(2^n-2)/3$ (and, concurrently, the number of sequences such that $\sum_{i=1}^n x_i=1,2$, equals $(2^n+1)/3$).

As a result, we conclude that the number of Heawood vectors equals $(2^n+2)/3+2$ for even $n$ and it does $(2^n-2)/3$ for odd~$n$. By Theorem~\ref{th:heawood} for calculating the number of Tait colorings we should multiply this value by~3.
\end{proof}

\section{Conclusion}
In the approach proposed by P.~J.~Heawood, the Tait coloring problem is reduced to the search of an everywhere nonzero solution to a system of $n+1$ linear equations in ${\mathbb F}_3^{2n}$. Equations in this system have a simple geometric sense. This allows one to prove that the rank of the system equals $n+1$ in a non-bipartite case and to establish a visual criterion for describing the set of free variables of the system. Correspondingly, one can easily deduce the inequality for the cardinality of the defining set of vertices. 
However, the requirement that a solution should be everywhere nonzero restricts the application of standard methods of linear algebra.

The graph $CL_3$ (the unique planar biconnected cubic graph with 6 vertices) considered at the end part of the Introduction illustrates the fact that the estimate $\chi’_3(G)\leq  3\cdot 2^{n-1}$ is exact for any~$n$. One can easily strengthen it, provided that $m(G)$ {\it the cardinality of a minimal defining set} in the planar graph~$G$ is known. 
Really, 
from Theorem~\ref{th:heawood} 
we get the estimate $\chi’_3(G)\leq  3\cdot 2^{m(G)}$. However, this estimate neither is exact for any graph~$G$, because $\chi’_3(G)/3$ is not necessarily a power of $2$. Nevertheless, we can assume that for any~$n$ there exists a planar biconnected cubic graph~$G$ with~$2n$ vertices such that $\chi’_3(G)= 3\cdot 2^{m(G)}$.

According to traditions related to the proof of the existence of a Tait coloring, it makes sense to study possible counterexamples of graphs with minimal numbers of vertices in order to prove their absence. Such graphs are said to be {\it irreducible}. The considered interpretation of the Tait coloring evidently implies that the mentioned graph can contain no triangular face. 
Really, the propriety condition implies that all its vertices should have spins of one and the same sign. We can contract vertices to one point with the spin of the opposite sign. 
We can easily prove that the propriety condition for faces of the new graph and the initial one turn into each other.

However, we can easily prove the absence of triangles in the nonreducible graph, as well as four-sided faces~(see, for example,~\cite{Ore}). It is more interesting to learn how to ``delete'' faces with 5 sides, which necessarily exist in an irreducible graph. There also exists a correspondence between the set of values of spins in the graph with a concrete five-sided face (let us denote it as $ABCDE$) and without such a face. We can establish this correspondence in the following way. Without loss of generality, we can assume that if the face $ABCDE$ is proper, then the spin that corresponds to the vertex $C$ equals $-1$, while for all the rest vertices the spin value equals $+1$. Let us delete the edge $AE$ from this graph and contract together vertices $A$ and $B$, as well as vertices $D$ and~$E$. Let us define the spin of $-1$ for new (contracted) vertices. Let us concurrently change the spin value for the vertex~$C$ for the opposite one, i.e., $+1$. Then we get a new cubic graph, all whose faces are proper (one can easily verify this property). However, in this case, as distinct from the case of a triangular face, the correspondence between sets of spin values is not biunique. Really, spins at contracted vertices in the new graph have one and the same sign, while the spin at the vertex $C$ located between them has the opposite sign. However, the latter property is not necessarily fulfilled, if the new graph allows the existence of the Heawood vector.

Possibly, the approach, which implies the use of the Fourier transform of the Kronecker $\delta$-function (that indicates the fulfillment of condition~\eqref{eq:main_eq}) and the application of modified methods for the $\alpha$-representation (see~\cite{ejc}), will prove to be more efficient. We are going to try it out in future.

\textbf{Remark.} After completing this paper, the author has noticed the work~\cite{ivanov} by M.~P.~Ivanov. In the mentioned paper, in particular, similarly to Theorem~\ref{th:Cln}, M.~P.~Ivanov calculates the value $\chi'_3(CL_n)$. To this end, he uses another technique, which is more direct but more difficult (it requires the introduction of the notion of a ``garland''). He also calculates that $\chi'_3(M_{2n})=2^n+4$, where $M_{2n}$ is the graph of the so-called M\"obius ladder, which differs from the graph $CL_n$ in two edges (in denotations introduced in item~2.4, edges $w_nw_1$ and $v_nv_1$ in the graph $CL_n$ are replaced with edges $w_nv_1$ and $v_nw_1$). Moreover, Ivanov proves that in the class of simple cubic graphs~$G$ with $2n$ edges ($n>2$) the maximal value of $\chi'_3(G)$ with even $n$ is attained for $G=CL_n$, while with odd $n$ it does for $G=M_{2n}$.
These graphs are bipartite; moreover, the latter one is non-planar. This property is much more informative than the estimate obtained by us in Corollary~\ref{concl}.

\end{document}